\documentclass[11pt]{amsart}

\issueinfo{00}{0}{Month}{Year}
\copyrightinfo{2008}{University of Calgary}
\pagespan{1}{2}
\PII{ISSN 1715-0868}
\include{cdmstdcmds}

\usepackage{amsmath,amsfonts,amssymb,amscd,amsthm}

\usepackage{latexsym}
\usepackage[T2A]{fontenc}
\usepackage[cp1251]{inputenc}
\usepackage{graphicx}

\theoremstyle{plain} 
\newtheorem{theorem}{Theorem}[section]
\newtheorem{lemma}[theorem]{Lemma}
\newtheorem{corollary}[theorem]{Corollary}
\newtheorem{proposition}[theorem]{Proposition}
\newtheorem{conj}[theorem]{Conjecture}
\theoremstyle{definition} 
\newtheorem{deff}[theorem]{Definition}
\theoremstyle{remark} 
\newtheorem{remark}[theorem]{Remark}

\newcommand{\bbR}{\mathbb{R}}

\newcommand{\bbE}{\mathbb{E}}
\newcommand{\bfp}{\mathbf{p}}
\newcommand{\bfq}{\mathbf{q}}
\newcommand{\bfr}{\mathbf{r}}

\newcommand{\Vol}{\mathrm{Vol}}

\begin{document}

\title{Kneser-Poulsen conjecture for a small number of intersections}

\author{Igors Gorbovickis}
\address{Department of Mathematics, University of Toronto, Room 6290, 40~St. George Street, Toronto, Ontario, Canada M5S~2E4}
\email{igors.gorbovickis@utoronto.ca}

%

\date{(June 23, 2010), and in revised form (date2).}
\subjclass[2000]{}

\thanks{Research supported in part by NSF Grant No. DMS–0209595 (USA)}

\begin{abstract}
The Kneser-Poulsen conjecture says that if a finite collection of balls in the Euclidean space $\bbE^d$ is rearranged so that the distance between each pair of centers does not get smaller, then the volume of the union of these balls also does not get smaller. In this paper we prove that if in the initial configuration the intersection of any two balls has common points with no more than $d+1$ other balls, then the conjecture holds.
\end{abstract}

\maketitle

\section{Notation}
Given a positive integer $d$, by $\bbE^d$ we will denote a $d$-dimensional Euclidean space. 
Let $\bfp= (\bfp_1,\ldots , \bfp_N)$ and $\bfq= (\bfq_1,\ldots , \bfq_N)$ be two configurations of $N$ points, where each $\bfp_i\in\bbE^d$ and each $\bfq_i\in\bbE^d$. Let $|\ldots|$ be the Euclidean norm. If for all $1\le i<j\le N$, $|\bfp_i-\bfp_j|\le |\bfq_i-\bfq_j|$, we say that $\bfq$ is an \textit{expansion} of $\bfp$ and $\bfp$ is a \textit{contraction} of $\bfq$. 
If $\bfp_0\in\bbE^d$, we denote by $B_d(\bfp_0, r)$ the closed $d$-dimensional ball of radius $r$ in $\bbE^d$ about the point $\bfp_0$. We define $B_d(\bfp_0, r)$ to be an empty set, if $r<0$ or if $r$ is not a real number. We also let $\Vol_d$ represent the $d$-dimensional volume.

\section{Introduction}
The following conjecture was independently stated by Kneser~\cite{Kneser} in 1955 and Poulsen~\cite{Poulsen} in 1954 for the case when $r_1=\dots =r_N$:

\begin{conj} \label{KP1}
If $\bfq= (\bfq_1,\ldots , \bfq_N)$ is an expansion of $\bfp= (\bfp_1,\ldots , \bfp_N)$ in $\bbE^d$, then for any vector of radii $\bfr=(r_1,\dots,r_N)$,
\begin{equation}\label{KPineq}
\Vol_d\left[\bigcup_{i=1}^N B_d(\bfp_i, r_i)\right] \le \Vol_d\left[\bigcup_{i=1}^N B_d(\bfq_i, r_i)\right].
\end{equation}
\end{conj}

For $d=1$, Conjecture~\ref{KP1} is obvious. In the case when $d=2$, it was proved by K.~Bezdek and R.~Connelly in~\cite{BezdekConnelly}, and for $d\ge 3$ the conjecture currently remains open.
References to related results as well as the history of this conjecture can be found in the same paper~\cite{BezdekConnelly}.
In the current paper we prove the following theorem which confirms Conjecture~\ref{KP1} for $d\ge 3$ under some additional assumptions:

\begin{theorem}\label{MainTh}
Consider a configuration of $N$ closed $d$-dimensional balls defined by their centers $\bfp= (\bfp_1,\ldots , \bfp_N)$ in $\bbE^d$ and corresponding radii $\bfr=(r_1,\dots,r_N)$. Assume that the intersection of every pair of these balls has common points with no more than $d+1$ other balls from the considered configuration. Then for any expansion $\bfq= (\bfq_1,\ldots , \bfq_N)\subset\bbE^d$ of the centers $\bfp$, inequality~(\ref{KPineq}) holds.
\end{theorem}
As the limiting case of this theorem, we get the following corollary:
\begin{corollary}\label{MainCol}
Consider a configuration of $N$ closed $d$-dimensional balls defined by their centers $\bfp= (\bfp_1,\ldots , \bfp_N)$ in $\bbE^d$ and corresponding radii $\bfr=(r_1,\dots,r_N)$. Assume that the intersection of every pair of these balls has common \emph{interior} points with no more than $d+1$ other balls from the considered configuration. Then for any expansion $\bfq= (\bfq_1,\ldots , \bfq_N)\subset\bbE^d$ of the centers $\bfp$, inequality~(\ref{KPineq}) holds.
\end{corollary}
Corollary~\ref{MainCol} can be viewed as a generalization of the following theorem proved in~\cite{BezdekConnelly}: Conjecture~\ref{KP1} holds, if $N\le d+3$.

The proof of Theorem~\ref{MainTh} implements the following general idea which can also be found in other works on related subjects, such as~\cite{Alexander85},~\cite{BezdekConnelly} and~\cite{CapPach}.
Namely, we embed the Euclidean space $\bbE^d$ in a higher dimensional space and instead of considering $d$-dimensional ball configurations, we consider corresponding higher dimensional objects. Viewing $\bfp$ and $\bfq$ as point configurations in a higher dimensional space, allows us to consider a piecewise smooth monotone expansion from $\bfp$ to $\bfq$.
At the same time the higher dimensional ball configurations still carry some information about the $d$-dimensional ones. It appears that under the assumption of Theorem~\ref{MainTh} we can use this information to obtain inequality~(\ref{KPineq}).

Since we will work in spaces of different dimensions, it will be convenient for the rest of the paper to fix $d$ as in Theorem~\ref{MainTh}. Due to the above mentioned result~\cite{BezdekConnelly} of Bezdek and Connelly, we can assume that $d\ge 3$. We will use letter $n$ to denote the dimension of an object in case if we want to emphasize that this dimension is not necessarily equal to $d$.

\textbf{Acknowledgments}: The author would like to thank the anonymous referee for pointing out the inaccuracies in the first version of the text.

\section{Voronoi regions}
In this subsection we first recall definitions of truncated Voronoi regions and the walls between them. Then we formulate Csik\'os's formula (Theorem~\ref{Csikos}).

Let $\bfp= (\bfp_1,\ldots , \bfp_N)$ be a configuration of points in $\bbE^n$ with balls of radii $\bfr=(r_1,\dots, r_N)$ centered at corresponding points of the configuration. 
The following sets are called \emph{(extended) Voronoi regions}:
\begin{equation}\label{VorDef}
C_{n,i}(\bfp,\bfr)=\{\bfp_0\in\bbE^n\mid\text{for all \,} j, |\bfp_0-\bfp_i|^2 - r_i^2\le |\bfp_0-\bfp_j|^2 - r_j^2  \}.
\end{equation}
\begin{remark}\label{VorRem}
It is easy to check that each Voronoi region $C_{n,i}(\bfp, \bfr)$ is a convex polyhedral set, and all of them together tile the whole Euclidean space $\bbE^n$. 
\end{remark}
We consider \emph{truncated Voronoi regions} $\hat C_{n,i}(\bfp, \bfr)=B_n(\bfp_i, r_i)\cap C_{n,i}(\bfp,\bfr)$ and for each pair of distinct indexes $i\neq j$, we define the \emph{wall} between two truncated Voronoi regions as $W_{n-1,ij}(\bfp, \bfr)=\hat C_{n,i}(\bfp, \bfr)\cap \hat C_{n,j}(\bfp, \bfr)$. Figure~\ref{pic2} gives an example of the truncated Voronoi region decomposition of the union of balls. The common boundaries of the shaded regions are the walls between corresponding truncated Voronoi regions.

\begin{figure}
\caption{The truncated Voronoi region decomposition of the union of balls}\label{pic2}
\includegraphics[width=60mm]{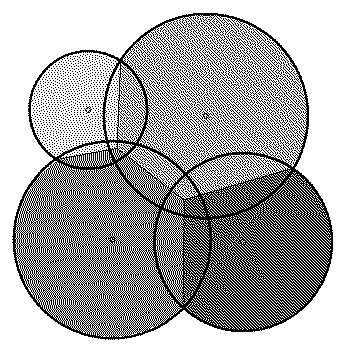}%
\end{figure}

We define the function $V_n(\bfp, \bfr)$ to be the volume of the union of balls from the ball configuration determined by $\bfp$ and $\bfr$
\begin{equation}\label{V_dDef}
V_n(\bfp, \bfr)=\Vol_n\left[\cup_{i=1}^N B_n(\bfp_i, r_i)\right].
\end{equation}
According to Remark~\ref{VorRem}, the function $V_n(\bfp, \bfr)$ can also be expressed as
$$
V_n(\bfp, \bfr)=\sum_{i=1}^N \Vol_n\left[\hat C_{n,i}(\bfp, \bfr)\right].
$$

Consider a smooth (infinitely many times differentiable) motion $\bfp(t)= (\bfp_1(t),\dots, \bfp_N(t))$ of some configuration of $N$ points in $\bbE^n$. Let $d_{ij}(t)= |\bfp_i(t)-\bfp_j(t)|$, and let $d_{ij}'$ be the $t$-derivative of $d_{ij}$. The following is Csik\'os's formula~\cite{Csikos98} for the $t$-derivative of the function $V_n(\bfp(t), \bfr)$.

\begin{theorem}\label{Csikos}
Let $n\ge 2$ and let $\bfp(t)$ be a smooth motion of a configuration of points in $\bbE^n$ such that for each $t$, all the points are pairwise distinct. Then the function $V_n(\bfp(t), \bfr)$ is differentiable with respect to $t$ and,

\begin{equation*}
\frac{d}{dt}V_n(\bfp(t), \bfr) =\sum_{1\le i<j\le N} d_{ij}'\Vol_{n-1}\left[W_{n-1,ij}(\bfp, \bfr)\right].
\end{equation*}
\end{theorem}

\section{Volume of a polyhedral set intersected with a ball}
We notice that both truncated Voronoi regions and the walls between them can be viewed as intersections of some polyhedral sets with corresponding balls. In this section we give relevant statements about the volumes of such sets. It appears that the volume of such a set can be obtained from the volume of a certain higher dimensional polyhedral set intersected with a ball. This will play an important role in our argument.

The following lemma is a reformulated Corollary~6 from~\cite{BezdekConnelly}.
\begin{lemma}\label{Arhimed}
Let $P\subset\bbE^{n+2}$ be a polyhedral set, such that all its codimension~1 facets are orthogonal to some $n$-dimensional affine subspace $X\subset\bbE^{n+2}$. Consider a point $\bfp_0\in X$. Then for every pair of real numbers $r$ and $s$, the following derivative exists and 
$$
\frac{d}{ds}\Vol_{n+2}\left[P\cap B_{n+2}(\bfp_0, \sqrt{r^2+s})\right]= \pi\Vol_n\left[ X\cap P\cap B_{n+2}(\bfp_0, \sqrt{r^2+s})\right].
$$
\end{lemma}

The following corollary can be proved by applying Lemma~\ref{Arhimed} several times.
\begin{corollary}\label{ArhimedCol}
Given a positive integer $k$, let $P\subset\bbE^{n+2k}$ be a polyhedral set, such that all its codimension~1 facets are orthogonal to some $n$-dimensional affine subspace $X\subset\bbE^{n+2k}$. Consider a point $\bfp_0\in X$. Then for every pair of real numbers $r$ and $s$, the following derivative exists and 
\begin{multline*}
\frac{d^k}{ds^k}\Vol_{n+2k}\left[P\cap B_{n+2k}(\bfp_0, \sqrt{r^2+s})\right]=\\
 \pi^k\Vol_n\left[ X\cap P\cap B_{n+2k}(\bfp_0, \sqrt{r^2+s})\right].
\end{multline*}
\end{corollary}


Given a positive integer $n$ and a non-negative integer $k$, let $\mathcal P_n^k$ be the space of all polyhedral sets in $\bbE^{n+2k}$ that are intersections of some $n$ half-spaces. We define topology on $\mathcal P_n^k$ in the following way. 
For a sufficiently small $\varepsilon>0$, an $\varepsilon$-neighborhood of a polyhedral set $P\in\mathcal P_n^k$ consists of all $P'\in\mathcal P_n^k$, such that $P$ and $P'$ can be represented as the intersection of $n$ half-spaces $H_1,\dots,H_n$ and $H_1',\dots,H_n'$ respectively, such that for each $i=1,\dots,n$, the hyperplanes $\partial H_i$ and $\partial H_i'$ are $\varepsilon$-close in some fixed metric on the affine Grassmannian $Graff(n+2k-1,n+2k)$ and the intersection $H_i\cap H_i'$ is either a half-space or a polyhedral set with an obtuse angle between its codimantion~$1$ facets.


\begin{lemma}\label{bound_l}
For $n\ge 3$, consider a polyhedral set $P\in\mathcal P_n^k$ and a point $\bfp_0\in\bbE^{n+2k}$. 
Then for every pair of real numbers $r$ and $s$, the following derivative exists and satisfies the inequality
\begin{multline}\label{bounds1}
0\le \frac{d^{k+1}}{ds^{k+1}}\Vol_{n+2k}\left[P\cap B_{n+2k}(\bfp_0, \sqrt{r^2+s})\right]\le \\
\max\left\{\frac{1}{2}\pi^k\sigma_{n-1} (r^2+s)^{\frac{n-2}{2}}, 0\right\},
\end{multline}
where $\sigma_{n-1}$ is the $(n-1)$-dimensional surface volume of the $n$-dimensional unit ball.
Moreover, the derivative $\frac{d^{k+1}}{ds^{k+1}}\Vol_{n+2k}\left[P\cap B_{n+2k}(\bfp_0, \sqrt{r^2+s})\right]$ depends continuously on $P$ and $s$ simultaneously.
\end{lemma}
\begin{proof}
Since $P\in\mathcal P_n^k$, it can be represented as $P=\cap_{i=1}^n H_i$ for some half-spaces $H_i\subset\bbE^{n+2k}$, hence there exists an $n$-dimensional affine subspace $X$ that contains the point $\bfp_0$ and is orthogonal to the boundary hyperplanes of the half-spaces $H_1,\dots, H_n$. Then Corollary~\ref{ArhimedCol} implies that 
\begin{multline*}
\frac{d^{k+1}}{ds^{k+1}}\Vol_{n+2k}\left[P\cap B_{n+2k}(\bfp_0, \sqrt{r^2+s})\right] = \\
\pi^k \frac{d}{ds}\Vol_n\left[ X\cap P\cap B_{n+2k}(\bfp_0, \sqrt{r^2+s})\right].
\end{multline*}

Now according to the chain rule, 
\begin{multline}\label{chain_r}
\frac{d^{k+1}}{ds^{k+1}}\Vol_{n+2k}\left[P\cap B_{n+2k}(\bfp_0, \sqrt{r^2+s})\right] = \\
\frac{\pi^k}{2\sqrt{r^2+s}} \left.\frac{d}{d\tilde r}\Vol_n\left[ X\cap P\cap B_{n+2k}(\bfp_0, \tilde r)\right]\right|_{\tilde r=\sqrt{r^2+s}}.
\end{multline}
Note that $X\cap P\cap B_{n+2k}(\bfp_0, \tilde r)$ is the intersection of an $n$-dimensional polyhedral set $X\cap P$ with a ball of radius $\tilde r$. The derivative of the volume of this set with respect to $\tilde r$ 
is equal to the surface volume of the spherical part of its boundary. Since this surface volume is non-negative and not greater than the surface volume of the $n$-dimensional ball of radius $\tilde r$, we obtain the required inequalities~(\ref{bounds1}).

Finally, we notice that the surface volume of the spherical part of the boundary considered in the previous paragraph, depends continuously on $P\in \mathcal P_n^k$ and $s$ simultaneously, hence according to~(\ref{chain_r}), when $r^2+s\neq 0$, the derivative $\frac{d^{k+1}}{ds^{k+1}}\Vol_{n+2k}\left[P\cap B_{n+2k}(\bfp_0, \sqrt{r^2+s})\right]$ also depends continuously on $P$ and $s$ simultaneously.

On the other hand, when $r^2+s$ approaches zero, both the lower and the upper bounds in~(\ref{bounds1}) approach zero as well, which implies that when $r^2+s=0$, the derivative $\frac{d^{k+1}}{ds^{k+1}}\Vol_{n+2k}\left[P\cap B_{n+2k}(\bfp_0, \sqrt{r^2+s})\right]$ is also continuous in $P$ and $s$ simultaneously.
\end{proof}

\section{The volumes of the walls and their derivatives}
We return to the original setting where we have a configuration of $N$ balls of corresponding radii $r_1,\dots, r_N$. 

\begin{deff}
Given a vector of radii $\bfr=(r_1,\dots,r_N)\in\bbR^N$, we will say that a configuration of $N$ points $\bfp=(\bfp_1,\dots,\bfp_N)$ in some Euclidean space $\bbE^n$ is $(d,\bfr)$-\emph{nice}, if in the configuration of $N$ balls of radii $r_1,\dots,r_N$ centered at the corresponding points $\bfp_1,\dots,\bfp_N$ the intersection of each pair of balls has common points with no more than $d+1$ other balls.
\end{deff}
\begin{deff}
We will also say that a configuration of $N$ points $\bfp=(\bfp_1,\dots,\bfp_N)$ in some Euclidean space $\bbE^n$ (where $n\ge N-1$) is \emph{in a strictly general position} if the points of $\bfp$ are vertexes of a non-degenerate $(N-1)$-simplex.
\end{deff}

From now on we fix the radii $\bfr=(r_1,\dots,r_N)$ and we consider a one parameter family
\begin{equation*}
\bfr(s)=\left(\sqrt{r_1^2+s},\dots,\sqrt{r_N^2+s}\right)
\end{equation*}
which coincides with the initial vector of radii $\bfr$, when $s=0$. 
\begin{proposition}\label{CellProp}
Let $\bfp=(\bfp_1,\dots,\bfp_N)\subset\bbE^n$ be a configuration of $N$ points. Then

(i) Each Voronoi region $C_{n,i}(\bfp,\bfr(s))$ is a convex polyhedral set completely determined by $\bfp$ and $\bfr$ and independent of $s$.

(ii) Each truncated Voronoi region $\hat C_{n,i}(\bfp, \bfr(s))$ is an intersection of a fixed convex polyhedral set from part~(i) and the ball $B_n(\bfp_i,\sqrt{r_i^2+s})$.

(iii) Each wall between truncated Voronoi regions $W_{n-1,ij}(\bfp,\bfr(s))$ is an intersection of the ball $B_n(\bfp_i,\sqrt{r_i^2+s})$ with an $(n-1)$-dimensional convex polyhedral set independent from $s$ and lying in the radical hyperplane of the balls $B_n(\bfp_i,r_i)$ and $B_n(\bfp_j,r_i)$.
\end{proposition}
\begin{proof}
(i) As it was noticed in Remark~\ref{VorRem}, the Voronoi region $C_{n,i}(\bfp,\bfr(s))$ is a convex polyhedral set. Its independence from $s$ follows from its definition~(\ref{VorDef}).

Parts (ii) and (iii) immediately follow from part~(i).
\end{proof}

Now we prove our key lemma:
\begin{lemma}\label{ContLemma}
(i) Let $d\ge 2$, and let $k$ be a nonnegative integer. 
Consider a $(d,\bfr)$-nice configuration of $N$ points $\bfp=(\bfp_1,\dots,\bfp_N)$ in $\bbE^{d+2k}$. Then for every pair of distinct indexes $i\neq j$ the $(d+2k-1)$-dimensional volume of the wall between truncated Voronoi regions $\Vol_{d+2k-1}\left[W_{d+2k-1,ij}(\bfp, \bfr(s))\right]$ is at least $k$ times differentiable as a function of $s$ in a sufficiently small neighborhood $U$ of the point $s=0$. Also for each $s\in U$, the 
partial derivatives
\begin{equation}\label{DWalls}
\frac{\partial^k}{\partial s^k}\Vol_{d+2k-1}\left[W_{d+2k-1,ij}(\bfp, \bfr(s))\right]
\end{equation}
are non-negative and locally bounded as functions of $(\bfp,s)$.

(ii) If in addition to the conditions of the first part the configuration $\bfp$ is in a strictly general position, then the partial derivatives~(\ref{DWalls}) 
are locally continuous in $\bfp$ and $s$ simultaneously.
\end{lemma}

\begin{proof}
For $k=0$ the Lemma is obvious, so further we will assume that $k>0$.

We denote by $H$ the radical hyperplane of 
the balls $B_{d+2k}(\bfp_i, r_i)$ and $B_{d+2k}(\bfp_j, r_j)$. Then according to part~(iii) of Proposition~\ref{CellProp}, hyperplane $H$ contains the wall $W_{d+2k-1,ij}(\bfp, \bfr(s))$ which corresponds to these two balls. Let the point $\bfp_0$ be the orthogonal projection of the point $\bfp_i$ onto the hyperplane $H$ and define $h=|\bfp_i-\bfp_0|$.

Since the point configuration $\bfp$ is $(d,\bfr)$-nice and the set of $(d,\bfr)$-nice point configurations is open, there exists a neighborhood of the origin $U\subset\bbR$, such that for all $s\in U$ the configuration $\bfp$ is $(d,\bfr(s))$-nice. This implies that for all $s\in U$, 
the wall $W_{d+2k-1,ij}(\bfp, \bfr(s))$ can be viewed as the intersection of the ball $B(\bfp_0,\sqrt{r_i^2-h^2+s})$  
with $d+1$ half-spaces $H_1,\dots,H_{d+1}$ in $H$. This observation together with Lemma~\ref{bound_l} proves part~(i) of Lemma~\ref{ContLemma}.

We notice that if the configuration $\bfp$ is in a strictly general position, then both the hyperplane $H$ and the half-spaces $H_1,\dots,H_{d+1}$ depend locally continuously on $\bfp$. Thus Lemma~\ref{bound_l} implies part~(ii) of Lemma~\ref{ContLemma}.
\end{proof}

\section{A path between $\bfp$ and $\bfq$}

Now the proof of Theorem~\ref{MainTh} is essentially based on choosing an appropriate piecewise smooth path in the space of sufficiently high dimension that connects the configurations $\bfp$ and $\bfq$. More detailed arguments follow.

\begin{lemma}\label{PathLemma}
If $\bfp(t)= (\bfp_1(t),\dots, \bfp_N(t))$ is a piecewise smooth motion of a configuration of centers in $\bbE^{d+2k}$ with $d\ge 2$ and $t\in[0,1]$, such that $\bfp(t)$ is $(d,\bfr)$-nice for all $t\in[0,1]$ and is in a strictly general position for all but finitely many values of $t$ in $[0,1]$, then the following identity holds: 
\begin{multline}\label{PathMult}
\left.\frac{\partial^k}{\partial s^k}\left(V_{d+2k}(\bfp(1), \bfr(s)) - V_{d+2k}(\bfp(0), \bfr(s))\right)\right|_{s=0}=\\
 \int_0^1 \sum_{1\le i<j\le N} d_{ij}'\left.\frac{\partial^k}{\partial s^k}\Vol_{d+2k-1}\left[W_{d+2k-1,ij}(\bfp(t), \bfr(s))\right]\right|_{s=0}dt,
\end{multline}
where the function $V_{d+2k}$ is defined as in~(\ref{V_dDef}) and $d_{ij}(t)=|\bfp_i(t)-\bfp_j(t)|$.
\end{lemma}
\begin{proof}
It is obvious that
\begin{multline*}
\left.\frac{\partial^k}{\partial s^k}\left(V_{d+2k}(\bfp(1), \bfr(s)) - V_{d+2k}(\bfp(0), \bfr(s))\right)\right|_{s=0}=\\ 
 \left.\frac{\partial^k}{\partial s^k}\int_0^1 \frac{\partial}{\partial t}V_{d+2k}(\bfp(t), \bfr(s)) dt\right|_{s=0}.
\end{multline*}
Now according to Csik\'os's formula (Theorem~\ref{Csikos}) we get that
\begin{multline*}
\left.\frac{\partial^k}{\partial s^k}\int_0^1 \frac{\partial}{\partial t}V_{d+2k}(\bfp(t), \bfr(s)) dt\right|_{s=0}=\\
 \left.\frac{\partial^k}{\partial s^k}\int_0^1 \sum_{1\le i<j\le N} d_{ij}'\Vol_{d+2k-1}\left[W_{d+2k-1,ij}(\bfp(t), \bfr(s))\right]dt\right|_{s=0}.
\end{multline*}
Finally, it follows from Lemma~\ref{ContLemma} that we can change the order of differentiation and integration in the last expression:
\begin{multline*}
\left.\frac{\partial^k}{\partial s^k}\int_0^1 \sum_{1\le i<j\le N} d_{ij}'\Vol_{d+2k-1}\left[W_{d+2k-1,ij}(\bfp(t), \bfr(s))\right]dt\right|_{s=0}=\\
 \int_0^1 \sum_{1\le i<j\le N} d_{ij}'\left.\frac{\partial^k}{\partial s^k}\Vol_{d+2k-1}\left[W_{d+2k-1,ij}(\bfp(t), \bfr(s))\right]\right|_{s=0}dt.
\end{multline*}
\end{proof}

\begin{corollary}\label{PerturbCol}
If $d\ge 2$ and $\bfp\subset\bbE^{d+2k}$ is a $(d,\bfr)$-nice configuration of $N$ points, where $N\le d+2k+1$, 
then the function $\left.\frac{\partial^k}{\partial s^k}V_{d+2k}(\bfp, \bfr(s))\right|_{s=0}$ is locally continuous in variable $\bfp$.
\end{corollary}
\begin{proof}
Since $\bfp$ is $(d, \bfr)$-nice, all point configurations that are sufficiently close to $\bfp$, are also $(d,\bfr)$-nice. If $\bfp'\subset\bbE^{d+2k}$ is a configuration of centers that is sufficiently close to $\bfp$, then we can connect the configurations $\bfp$ and $\bfp'$ with a piecewise smooth path $\bfp(t)$ that satisfies Lemma~\ref{PathLemma}. Then Corollary~\ref{PerturbCol} follows from the fact that the functions 
$$
\left.\frac{\partial^k}{\partial s^k}\Vol_{d+2k-1}\left[W_{d+2k-1,ij}(\bfp(t), \bfr(s))\right]\right|_{s=0}
$$
in the right hand side of~(\ref{PathMult}) are bounded, as was shown in Lemma~\ref{ContLemma}.
\end{proof}

\begin{corollary}\label{ExpCor}
If $d\ge 2$ and $\bfp,\bfq\subset\bbE^{d+2k}$ are two configurations of $N$ points that are in a strictly general position, $\bfq$ is an expansion of $\bfp$, and configuration $\bfp$ is $(d,\bfr)$-nice, then
$$
\left.\frac{\partial^k}{\partial s^k}V_{d+2k}(\bfq, \bfr(s))\right|_{s=0} \ge \left.\frac{\partial^k}{\partial s^k}V_{d+2k}(\bfp, \bfr(s))\right|_{s=0}.
$$
\end{corollary}
\begin{proof}
According to~\cite{Alexander85}, configurations $\bfp$ and $\bfq$ can be connected by a piecewise smooth motion $\bfp(t)$ so that $\bfp(0)=\bfp$, $\bfp(1)=\bfq$, the distances $d_{ij}(t)=|\bfp_i(t)-\bfp_j(t)|$ are weakly increasing in $t$ and for each $t$ the point configuration $\bfp(t)$ is in a strictly general position. It follows from Kirszbraun's theorem~\cite{Kirszbraun} that since the distances $d_{ij}(t)$ are weakly increasing and $\bfp(0)$ is $(d,\bfr)$-nice, then $\bfp(t)$ is $(d,\bfr)$-nice for all $t\in[0,1]$, and we can apply Lemma~\ref{PathLemma}:
\begin{multline*}
\left.\frac{\partial^k}{\partial s^k}(V_{d+2k}(\bfq, \bfr(s)) - V_{d+2k}(\bfp, \bfr(s)))\right|_{s=0}=\\
 \int_0^1 \sum_{1\le i<j\le N} d_{ij}'\left.\frac{\partial^k}{\partial s^k}\Vol_{d+2k-1}\left[W_{d+2k-1,ij}(\bfp(t), \bfr(s))\right]\right|_{s=0}dt.
\end{multline*}
Because of Lemma~\ref{ContLemma}, the derivatives 
$$
\left.\frac{\partial^k}{\partial s^k}\Vol_{d+2k-1}\left[W_{d+2k-1,ij}(\bfp(t), \bfr(s))\right]\right|_{s=0}
$$
are always non-negative, and since $d_{ij}'\ge 0$, the expression under the integral is also non-negative.
\end{proof}

\begin{proof}[Proof of Theorem~\ref{MainTh}]
Let $k$ be a positive integer, such that $d+2k\ge N-1$, and we regard $\bbE^d$ as the subset $\bbE^d=\bbE^d\times\{0\}\subset\bbE^d\times\bbE^{2k}=\bbE^{d+2k}$. We can view $\bfp$ and $\bfq$ as point configurations lying either in $\bbE^d$ or in $\bbE^{d+2k}$ and consider corresponding $d$-dimensional and $(d+2k)$-dimensional volumes $V_d(\bfp,\bfr)$, $V_d(\bfq,\bfr)$, $V_{d+2k}(\bfp,\bfr)$ and $V_{d+2k}(\bfq,\bfr)$.

Note that the sets $\bigcup_{i=1}^N B_{d+2k}(\bfp_i, r_i)$ and $\bigcup_{i=1}^N B_{d+2k}(\bfq_i, r_i)$ are disjoint unions of truncated Voronoi regions and according to part~(ii) of Proposition~\ref{CellProp}, we can apply Corollary~\ref{ArhimedCol} to them. As a result, we obtain the following identity:
\begin{equation}\label{DimEq}
\pi^k(V_d(\bfq, \bfr) - V_d(\bfp, \bfr))= \left.\frac{\partial^k}{\partial s^k}(V_{d+2k}(\bfq, \bfr(s)) - V_{d+2k}(\bfp, \bfr(s)))\right|_{s=0}.
\end{equation}

Because of Kirszbraun's theorem, since $\bfp$ is $(d,\bfr)$-nice and $\bfq$ is its expansion, $\bfq$ is also $(d,\bfr)$-nice. Hence according to Corollary~\ref{PerturbCol}, the right hand side of~(\ref{DimEq}) depends locally continuously on $\bfp$ and $\bfq$. Let $\bfp', \bfq'\subset\bbE^{d+2k}$ be small perturbations of $\bfp$ and $\bfq$ respectively, such that configurations $\bfp'$ and $\bfq'$ are in a strictly general position, $\bfp'$ is $(d,\bfr)$-nice and $\bfq'$ is an expansion of $\bfp'$. Then it follows from Corollary~\ref{ExpCor} that 
\begin{equation}\label{PrimeEq}
\left.\frac{\partial^k}{\partial s^k}(V_{d+2k}(\bfq', \bfr(s)) - V_{d+2k}(\bfp', \bfr(s)))\right|_{s=0}\ge 0.
\end{equation}
By choosing $\bfp'$ and $\bfq'$ arbitrarily close to $\bfp$ and $\bfq$ respectively, we get the following inequality as a limiting case of~(\ref{PrimeEq}):
$$
\left.\frac{\partial^k}{\partial s^k}(V_{d+2k}(\bfq, \bfr(s)) - V_{d+2k}(\bfp, \bfr(s)))\right|_{s=0}\ge 0.
$$
Together with~(\ref{DimEq}) this proves that
$$
V_d(\bfq, \bfr) \ge V_d(\bfp, \bfr).
$$
\end{proof}


\nocite{*}

\bibliographystyle{abbrv}
\bibliography{small_intersect}

\end{document}